\documentclass[10pt]{amsart}
\usepackage{amsmath,amssymb,amsthm}
\usepackage{url}
\usepackage{graphicx}
\usepackage{labelfig}
\newtheorem{theorem}{Theorem}[section]    
\newtheorem{lemma}[theorem]{Lemma}          
\newtheorem{proposition}[theorem]{Proposition}  

\newtheorem{corollary}[theorem]{Corollary} 

\theoremstyle{definition}

\newtheorem{remark}[theorem]{Remark}

\newtheorem{convention}[theorem]{Convention}

 \makeatletter
    
    \@addtoreset{equation}{section}
  \makeatother
\def\co{\colon \thinspace}

\newcommand{\GL}{\textrm{GL}}
\newcommand{\Z}{\mathbb{Z}}
\newcommand{\R}{\mathbb{R}}
\newcommand{\Q}{\mathbb{Q}}
\newcommand{\C}{\mathbb{C}}
\newcommand{\mL}{\mathbb{L}}
\newcommand{\mcL}{\mathcal{L}}
\newcommand{\sltwo}{\mathfrak{sl}_{2}}
\newcommand{\mU}{\mathcal{U}}
\newcommand{\sq}{\mathsf{q}}
\newcommand{\st}{\mathsf{t}}

\newcommand{\tC}{\widetilde{C}}
\newcommand{\bF}{\mathbb{F}}
\newcommand{\be}{\mathbf{e}}
\newcommand{\bk}{\mathbf{k}}
\newcommand{\bz}{\mathbf{z}}
\newcommand{\bS}{\mathbb{S}}
\newcommand{\fg}{\mathfrak{g}}

\title[Reading the dual Garside length]{Reading the dual Garside length of braids from homological and quantum representations}
\author{Tetsuya Ito}
\address{Department of Mathematics,
The University of British Columbia, 1984 Mathematics Road
Vancouver, B.C, Canada V6T 1Z2}
\email{tetitoh@math.ubc.ca}
\urladdr{http://ms.u-tokyo.ac.jp/~tetitoh/}
\subjclass[2010]{Primary~20F36 
, Secondary~20F10,57M07,20G42}
\keywords{Quantum group, Lawrence's representation, Braid group, Dual Garside length}

\begin{document}

\begin{abstract} 
We show that Lawrence's representation and linear representations from $U_{q}(\sltwo)$ called generic highest weight vectors detect the dual Garside length of braids in a simple and natural way. That is, by expressing a representation as a matrix over a Laurent polynomial ring using certain natural basis, the span of the variable is equal to the constant multiples of the dual Garside length.
\end{abstract}
\maketitle

\section{Introduction}

The braid group $B_{n}$ possesses nice combinatorial structures, called {\em Garside structures}. A Garside structure allows us to solve the word problem and the conjugacy problem, and gives an efficient method to compute the length function for some generating sets called {\em simple elements}. In this paper, we restrict our attention to one of the standard Garside structure of the braid groups found by Birman-Ko-Lee in \cite{bkl}, known as the {\em dual Garside structure}. We study a relationship between the {\em dual Garside length}, the length function associated to the dual Garside structure, and linear representations of braid groups. 

In \cite{iw}, the author and Wiest proved that one of the most famous linear representation, the {\em Lawrence-Krammer-Bigelow representation} (the LKB representation, in short) detects the dual Garside length in a simple and natural way, as was conjectured by Krammer \cite{kra1}.

The aim of this paper is to generalize this result for two infinite families of braid representations derived from different constructions. One family of representations is called {\em Lawrence's representations} constructed by Lawrence in \cite{law} in a topological method. The other family is derived from the quantum group $U_{q}(\sltwo)$, and called {\em generic highest weight vectors}. 
We will prove that these representations also detect the dual Garside length in a simple and natural way as the LKB representation does.

Let us explain what the phrase {\em ``in a simple and natural way''} means.
Let $R$ be a ring and $a \in R[x^{\pm 1}]$ be a Laurent polynomial of the variable $x$ whose coefficient is in $R$. We denote the maximal and the minimal degree of the variable $x$ by $M_{x}(a)$ and $m_{x}(a)$, respectively.
For an $(N \times M)$-matrix $A =(a_{\iota,\kappa})$ over $R[x^{\pm 1}]$, we define
\[ M_{x}(A)=\max_{ \iota,\kappa } \{ M_{x}(a_{\iota,\kappa})\}, \;\;\;\; m_{x}(A) = \min_{ \iota,\kappa}\{ m_{x}(a_{\iota,\kappa})\} . \] 

Now consider a linear representation 
\[ \rho: B_{n} \rightarrow \GL( V ). \]
where $V$ is a free $R[x^{\pm 1}]$-module.
There are many known constructions of such linear representations of braid groups. In many cases, from a construction of the linear representation $\rho$ one can find a certain natural basis $\{v_{i}\}_{i=1,2,\ldots,N}$ of $V$. By using the basis $\{v_{i}\}$, we express the representation $\rho$ as a matrix representation
\[ \rho^{V} \co B_{n} \rightarrow \GL( N ; R[x^{\pm 1}] ). \]
So for each braid $\beta$ we get an $(N \times N)$-matrix $\rho^{V}(\beta)$ over $R[x^{\pm 1}]$.
We say the representation $\rho^{V}$ {\em naturally} and {\em directly detects} the dual Garside length if the supremum and the infimum of $\beta$ are some constant multiples of $m_{x}=m_{x}(\rho^{V}(\beta))$ and $M_{x}=M_{x}(\rho^{V}(\beta))$, respectively. Here the supremum and infimum of braids are integers derived from the dual Garside structure which will compute the dual Garside length. Roughly speaking, $\rho^{V}$ naturally and directly detects the dual Garside length means that the dual Garside length is equal to some constant multiples of $M_{x}-m_{x}$, the span of the variable $x$.

Our results give a rather surprising and strong connection between linear representations and the dual Garside structure. Assume that we used other basis $\{v'_{i}\}$ of $V$. Then it is still true that one can compute the dual Garside length from the matrix expression of $\rho(\beta)$ with respect to the basis $\{v'_{i}\}$. However, for such basis to compute the dual Garside length it is not sufficient to know $m_{x}$ and $M_{x}$. We need to know the whole matrix, and the formula of dual Garside length might be quite complicated.
Thus our results says, roughly speaking, a natural basis of representation derived from a geometric or algebraic point of view is also natural with respect to the dual Garside structure.

The plan of this paper is as follows. In Section \ref{sec:dualGarside} we review a definition of dual Garside length. In Section \ref{sec:Lawrence} we briefly summarize Bigelow-like treatment of Lawrence's representation. A construction of quantum representations will be reviewed in Section \ref{sec:quantum}. We will also review Kohno's theorem which identifies Lawrence's representation with certain quantum representations, called {\em generic null vectors}.
Finally we will prove the main results of this paper, the dual Garside length formulae in Section \ref{sec:proof}.

Unfortunately, our proof of the dual Garside length formula for quantum representation, Theorem \ref{theorem:Q}, is indirect and tells us nothing why such an equality holds. Our proof is based on a topological method developed in \cite{iw} and we needed to use Theorem \ref{theorem:L}, the dual Garside length formula for Lawrence's representation. 

Contrary, the proof of Theorem \ref{theorem:L} partially provides a reason why Lawrence's representation detects the dual Garside length. As we have seen \cite{iw}, the dual Garside length are related to certain abelian covering (or, local coefficients) of punctured disc, which is a key ingredient of construction of Lawrence's representation. This observation and non-vanishing of the homological intersection pairing called Noodle-Fork pairing (Lemma \ref{lemma:key}) which was the key of the proof of faithfulness in Bigelow's theory lead to the dual Garside length.

It is an interesting problem to find an alternative proof of Theorem \ref{theorem:Q} via the theory of quantum groups. Our results suggest that there are unknown relationships between the dual Garside structure and quantum groups: in particular, the universal R-matrix of quantum groups might have a rich combinatorial structure, than the simple fact that it gives rise to a solution of Yang-Baxter equation.

It is also interesting to study a relationship between linear representations of braids and {\em usual Garside structure}, the other  standard Garside structure of braids. In \cite{kra2} Krammer proved that the LKB representation naturally and directly detects the usual Garside structure, by using the variable $\st$ instead of $\sq$. Hence it is natural to expect that Lawrence's representation and quantum representations which we considered here also detects the usual Garside structure, by using the other variables.

\begin{convention}
In a theory of the braid groups and quantum groups, various conventions are used. To avoid confusions, we summarize our conventions and notations below.

\begin{itemize}
\item All actions of the braid groups are considered as {\em left} actions, and the dual Garside structure we are treating is {\em right} Garside structure, as we will review in Section \ref{sec:dualGarside}.

\item All conventions about quantum groups $U_{q}(\sltwo)$ follows from \cite{jk}. In particular, $q$-numbers, $q$-fractionals, and $q$-binomial coefficients are defined by
\[ [n]_{q} ! = [n]_{q}[n-1]_{q} \cdots [2]_{q}[1]_{q},\]
\[ [n]_{q} = \frac{q^{n}-q^{-n}}{q-q^{-1}}, \;\; \;\; \left[ \!\!\begin{array}{c} n \\ j \end{array} \!\! \right]_{q} = \frac{[n]_{q}!}{[n-j]_{q}![j]_{q}!}. \]
As in \cite{jk}, we use Kassel's textbook \cite{kas} as a main reference of quantum groups. One main difference is that we use a variable (parameter) $\hbar$, which corresponds $2h$ in Kassel's book. 

\item As an element of mapping class group of $D_{n}$, a positive standard generator $\sigma_{i}$ of $B_{n}$ is identified with the {\em right-handed}, that is, the {\em unti-clockwise} half Dehn-twist which interchanges the punctures $p_{i}$ and $p_{i+1}$ of $D_{n}$.

It should be emphasized that this conventions is opposite to \cite{big1}, \cite{iw} and \cite{jk}. This convention is adapted to simplify the identification of quantum and homological representations, Corollary \ref{theorem:jkconjecture}.

\item The rest of conventions about homological representations, topological construction of representations we mainly follow \cite{iw}. In particular, the positive direction of winding (the meridian of hypersurfaces) is chosen as the {\em clockwise direction}. 

\item To distinguish the variables in quantum and Lawrence's representation, we use symbol $\sq,\st$ to represent variables in Lawrence's representations whereas we use $q,s$ to represent variables in quantum representations. Thus, $\sq$ and $q$ represent {\em different} variables.

\item Let $R$ be a subring of $\C$. For $R[x^{\pm 1}, y^{\pm 1}]$-module $V$ (where $x$ and $y$ are variables), we will denote the specialization of the variable $x$ and $y$ to complex numbers $c$ and $c'$ by $V|_{x=c, y=c'}$.
That is, $V|_{x=c,y=c'}$ is a $\C$-vector space $ \C \otimes_{R[x^{\pm 1},y^{\pm 1}]} V$ where we regard $\C$ as an $R[x^{\pm 1}, y^{\pm 1}]$-module by the specialization map $f_{c,c'} \co R[x^{\pm 1}, y^{\pm 1}] \rightarrow \C$ defined by $f_{c,c'}(x)=c$, $f_{c,c'}(y)=c'$.
\end{itemize}
\end{convention}

\textbf{Acknowledgments.} 
This research was supported by JSPS Research Fellowships for Young Scientists.
The author would like to thank Toshitake Kohno for stimulating discussion.

\section{Dual Garside length}
\label{sec:dualGarside}

In this section we review the dual Garside length.  
For details, see \cite{bkl}. 
For $1\leqslant i < j \leqslant n$, let $a_{i,j}$ be the braid
\[ a_{i,j} = (\sigma_{i+1} \cdots \sigma_{j-2}\sigma_{j-1})^{-1} \sigma_{i} ( \sigma_{i+1} \cdots \sigma_{j-2}\sigma_{j-1}) \] 

The generating set $\Sigma^{*} =\{a_{i,j}\: | \: 1\leqslant i < j \leqslant n\}$ was introduced in \cite{bkl}, and its elements are called the {\em dual Garside generators}, or {\em band generators}, or {\em Birman-Ko-Lee generators}.

A {\em dual-positive braid} is a braid which is written by a product of positive dual Garside generator $\Sigma^{*}$. The set of dual-positive braids is denoted by $B_{n}^{+*}$.
The {\em dual Garside element} is a braid $\delta$ given by
\[ \delta= a_{1,2}a_{2,3}\ldots a_{n-1,n} \]

Let $\preccurlyeq_{\Sigma^{*}}$ be the (right) subword partial ordering with respect to the dual Garside generating set~$\Sigma^{*}$: $\beta_{1} \preccurlyeq_{\Sigma^{*}} \beta_{2}$ if and only if $\beta_{2}\beta_{1}^{-1} \in B_{n}^{+*}$.
For a given braid $\beta$, the supremum $\sup_{\Sigma^{*}}(\beta)$ and the infimum $\inf_{\Sigma^{*}}(\beta)$ is defined by
\[ \sup\!{}_{\Sigma^{*}}(\beta) = \min \{ m \in \Z \: | \: \beta \preccurlyeq_{\Sigma^{*}} \delta^{m} \} \]
and 
\[ \inf\!{}_{\Sigma^{*}}(\beta) = \max \{ M \in \Z \: | \: \delta^{M} \preccurlyeq_{\Sigma^{*}} \beta \}\]
respectively.

A \emph{dual-simple} element is a dual-positive braid $x$ which satisfies $1 \preccurlyeq_{\Sigma^{*}} x \preccurlyeq_{\Sigma^{*}} \delta$. The set of dual-simple element is denoted by $[1,\delta]$.
The {\em dual Garside length} $l_{\Sigma^{*}}$ is the length function with respect to the generating set $[1,\delta]$. 

The next formula relates the supremum, infimum and the length.

\begin{proposition}[Dual Garside length \cite{bkl}]
\label{prop:dGlength}
For a braid $\beta \in B_{n}$ we have the following equality
\[ l_{\Sigma^{*}}(\beta) = \max\{0,\sup\!{}_{\Sigma^{*}}(\beta)\} - \min\{ \inf\!{}_{\Sigma^{*}}(\beta),0\}. \]
\end{proposition}

\section{Lawrence's representation}
\label{sec:Lawrence}
In this section we review basics of Lawrence's representation.

\subsection{Definition of Lawrence's representations}

First we review a construction of Lawrence's representation. See \cite{law}.
For $i= 1,2,\ldots,n$, let $p_{i}=i \in \C$ and $D_{n}=\{ z \in \C \: | \: |z| \leq n+1 \} - \{p_{1},\ldots,p_{n}\}$ be the $n$-punctured disc. The braid group $B_{n}$ is identified with the mapping class group of $D_{n}$, namely, the group of isotopy classes of homeomorphisms of $D_{n}$ that fix $\partial D_{n}$ pointwise. As we mentioned, the standard generator $\sigma_{i}$ is regarded as {\em right-handed} half Dehn twists which interchanges the $i$-th and $(i+1)$-st punctures.

For $m>0$ let $C_{n,m}$ be the unordered configuration space of $m$-points in $D_{n}$:
\[ C_{n,m} =\{ (z_{1},\ldots,z_{m}) \in D_{n} \: | \: z_{i} \neq z_{j} \;(i\neq j) \} \slash S_{m} \]
where $S_{m}$ is the symmetric group that acts as permutation of coordinates.
For $i=1,\ldots,n$, let $d_{i}=(n+1)e^{(\frac{3}{2}+ i\varepsilon)\pi \sqrt{-1}}  \in \partial D_{n}$ where $\varepsilon >0$ is sufficiently small number. We take $\{d_{1},\ldots,d_{n}\}$ as a base point of $C_{n,m}$.

The first homology group $H_{1}(C_{n,m}; \Z)$ is isomorphic to $\Z^{\oplus n} \oplus \Z$, where the first $n$ components corresponds to the meridians of the hyperplane $z_{1}=p_{i}$ $(i=1,\ldots,n)$ and the last component corresponds to the meridian of the discriminant, the union of hyperplanes $z_{i}=z_{j}$.
Let $\alpha: \pi_{1}(C_{n,m}) \rightarrow \Z^{2} = \langle \sq,\st\rangle$ be the homomorphism obtained by composing the Hurewicz homomorphism $\pi_{1}(C_{n,m}) \rightarrow H_{1}(C_{n,m}; \Z)$ and the projection $H_{1}(C_{n,m}; \Z) = \Z^{\oplus n} \oplus \Z \rightarrow \Z \oplus \Z = \langle \sq \rangle \oplus \langle \st \rangle$.

Let $\pi: \tC_{n,m} \rightarrow C_{n,m}$ be the covering corresponding to $\textrm{Ker}\,\alpha$, and 
fix a lift of the base point $\{\widetilde{d_{1}},\ldots,\widetilde{d_{m}}\}$. We regard $\sq$ and $\st$ as deck translations.
Then $H_{m}(\tC_{n,m};\Z)$ is a free $\Z[\sq^{\pm 1},\st^{\pm 1}]$-module of rank 
\[ d_{n,m} = \left( \begin{array}{c} m+n-2 \\ m \end{array}\right). \]

Let $E_{n,m}$ be the set 
\[ E_{n,m}= \{ \be = (e_{1},\ldots,e_{n-1}) \in \Z_{\geq 0}^{n-1}\: | \: e_{1} +\cdots +e_{n-1}=m\}\]
The cardinal of the set $E_{n,m}$ is equal to $d_{n,m}$.
We will use the set $E_{n,m}$ to give an index of basis.

The homomorphism $\alpha$ is invariant under the $B_{n}$-action so we get a linear representation 
\[ L'_{n,m}\co  B_{n} \rightarrow \GL(H_{m}(\tC_{n,m}; \Z) ). \]
We call this representation {\em Homological Lawrence's representation}.

There are subtle points in Lawrence's representation: As we will see later, to obtain an explicit matrix representation we use a certain basis of $H_{m}(\tC_{n,m};\Q)$ (as $\Q[\sq^{\pm1},\st^{\pm 1}])$-module) indexed by $E_{n,m}$, called the {\em standard multifork basis}. Unfortunately, the standard multifork basis might fail to be a basis of $H_{m}(\tC_{n,m};\Z)$ as $\Z[\sq^{\pm 1},\st^{\pm 1}]$-module (See \cite{pp} for the case $m=2$.). 
However, by using the standard multifork basis we will actually get a matrix representation of {\em integer} coefficients,
\[ L_{n,m} \co B_{n} \rightarrow \GL( d_{n,m}; \Z[\sq^{\pm 1},\st^{\pm 1}]). \]
We call $L_{n,m}$ {\em Lawrence's representation}, or, {\em Geometric Lawrence's representation}.

The case $m=1$ is rather special: In this case, it does not involve the variable $\st$ since the discriminant set is empty, and the Lawrence's representation $L_{n,1}$ is known to be the same as the reduced Burau representation.
The case $m=2$ was intensively studied by Bigelow \cite{big1} and Krammer \cite{kra1}, \cite{kra2}, and called the {\em Lawrence-Krammer-Bigelow representation} (the LKB representation).

\subsection{Multiforks and noodle-fork pairing}

In this section we give a brief exposition of generalization of Bigelow's theory of the LKB representation (the case $m=2$) to general Lawrence's representations. Such a generalization can be found in \cite{z1},\cite{z2}. The case $m=1$ was also treated by Bigelow in \cite{big4}.

To represents a homology classes of $H_{m}(\tC_{n,m};\Z)$, we use a geometric object, called {\em fork} and {\em multifork}. 

Let $Y$ be the $Y$-shaped graph shown in Figure~\ref{fig:multifork}, having one distinguished external vertex $r$, two other external vertices $v_{1}$ and $v_{2}$, and one internal vertex~$c$. We orient the edges of $Y$ as shown in Figure \ref{fig:multifork}.

A {\em fork} $F$ based on $d_{i}$ is an embedded image of $Y$ into $D^2$ such that:
\begin{itemize}
\item All points of $Y\setminus \{r,v_1,v_2\}$ are mapped to the interior of~$D_n$.
\item The distinguished vertex $r$ is mapped to $d_{i}$.
\item The other two external vertices $v_{1}$ and $v_{2}$ are mapped to two different puncture points.
\item The edge $[r,c]$ and the arc $[v_1,v_2]=[v_{1},c] \cup [c,v_{2}]$ are both mapped smoothly.
\end{itemize}

The image of the edge $[r,c]$ is called the {\em handle} of~$F$ and denoted $H(F)$. The image of $[v_1,v_2]=[v_{1},c] \cup [c,v_{2}]$, regarded as a single oriented arc, is called the {\em tine}, denoted $T(F)$. The image of $c$ is called the \emph{branch point} of~$F$.

A {\em multifork} is a family of forks $\bF= (F_{1},\ldots,F_{m})$ such that 
\begin{itemize}
\item $F_{i}$ is a fork based on $d_{i}$.
\item $T(F_{i}) \cap T(F_{j}) \cap D_{n} = \phi$ $(i\neq j)$.
\item $H(F_{i}) \cap H(F_{j}) = \phi$ $(i \neq j)$.
\end{itemize}
See the middle of Figure \ref{fig:multifork}, which gives an example of multifork for the case $m=3$.

\begin{figure}[htbp]
 \begin{center}
 \includegraphics[scale=0.5, width=100mm]{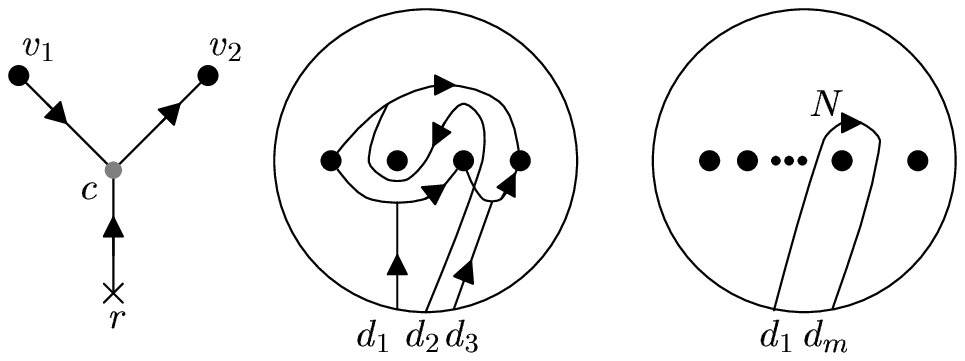}
  \caption{Multifork and Noodle}
 \label{fig:multifork}
  \end{center}
  \end{figure}

Let $\gamma_{i} \co [0,1] \rightarrow D_{n}$ be the handle of $F_{i}$, viewed as a path in $D_{n}$. Since we assumed that all handles are disjoint, we get a path in $C_{n,m}$,
\[ H(\bF) = \{\gamma_{1},\ldots,\gamma_{m}\} \co [0,1] \rightarrow C_{n,m}. \]
Take a lift of $H(\bF)$
\[ \widetilde{H(\bF)} \co [0,1] \rightarrow \tC_{n,m} \]
so that $\widetilde{H(\bF)}(0)=\{ \widetilde{d_{1}},\ldots, \widetilde{d_{m}}\}$.

Let $\Sigma(\bF) = \left\{ \{z_{1},\ldots,z_{m} \} \in C_{n,m} \: | \: z_{i} \in T(F_{i}) \right\}$ be the $m$-dimensional submanifold in $C_{n,m}$. Let $\widetilde{\Sigma}(\bF)$ be the $m$-dimensional submanifold of $\tC_{n,m}$ which is the connected component of $\pi^{-1}(\Sigma(\bF))$ that contains the point $\widetilde{H(\bF)}(1)$.
 Then $\widetilde{\Sigma}(\bF)$ defines an element of $H_{m}(\tC_{n,m};\Z)$. By abuse of notation, we will use $\bF$ to represent both multifork and the homology class $[\widetilde{\Sigma}(\bF)] \in H_{m}(\tC_{n,m};\Z)$.
 
Using a multifork, we construct a basis of $H_{m}(\tC_{n,m};\Q)$ as a $\Q[\sq^{\pm 1}, \st^{\pm 1}]$-module, indexed by the set $E_{n,m}$ as follows.
For $\be = (e_{1},\ldots,e_{n-1}) \in E_{n,m}$, we assign a multifork $\bF_{\be}=\{F_{1},\ldots, F_{m}\}$ as shown in Figure \ref{fig:standmultifork}. We say such a multifork a {\em standard multifork}. A standard multifork corresponding to a sequence of the form 
 \[ \be =(\underbrace{0,\ldots,0}_{i-1},m,\underbrace{0,\ldots,0}_{n-1-i}) \]
  is called a {\em straight fork} and denoted by $\bF_{i}$.
  
\begin{figure}[htbp]
 \begin{center}
\includegraphics[scale=0.5, width=90mm]{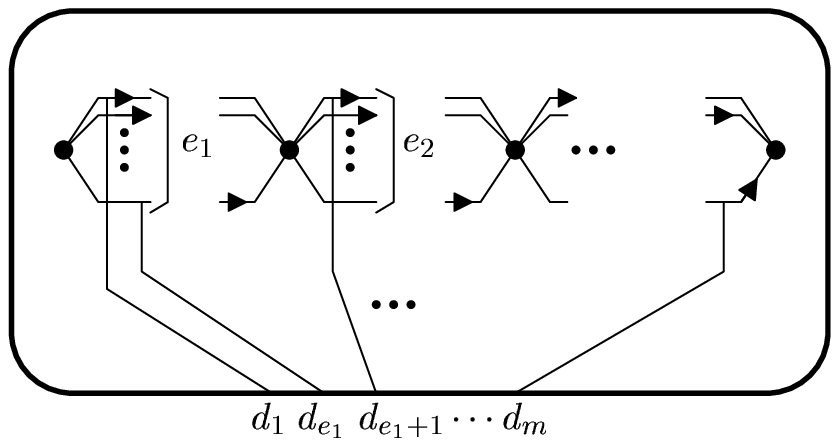}
 \caption{Standard multifork $\bF_{\mathbf{e}}$}
 \label{fig:standmultifork}
  \end{center}
\end{figure}

\begin{proposition}[Lawrence \cite{law}, Zheng \cite{z1}]
A set of standard multiforks $\{\bF_{\be}\}_{\be \in E_{n,m}}$ forms a basis of $H_{m}(\tC_{n,m};\Q)$ as a $\Q[\sq^{\pm 1},\st^{\pm 1}]$-module.
Moreover, for an $n$-braid $\beta$ and a standard multifork $\bF$, 
\[ \beta(\bF) \in \textrm{span}_{\Z[\sq^{\pm 1},\st^{\pm 1}]} \{ \bF_{\be} \}_{\be \in E_{n,m}} \subset H_{m}(\tC_{n,m};\Z). \]
\end{proposition}

Thus, by using the standard multifork basis we obtain a matrix representation
\[ L_{n,m} : B_{n} \rightarrow \GL(d_{n,m}; \Z[\sq^{\pm 1},\st^{\pm 1}]) \]
which we call {\em Lawrence's representation}.

\begin{remark}
In the case $m=2$ (the LKB representation case), the definition of standard (multi)fork is different from the definition used in \cite{big1},\cite{iw}. So as an explicit matrix representation, $L_{n,2}$ and the LKB representation given in \cite{big1} (and considered in \cite{iw}) are not the same. 
We also remark that the definition of the noodle-fork pairing which will be given below is also different from the definition of \cite{big1}.

However, as we will see in the Section \ref{sec:Lformula}, these difference does not affect the main results in \cite{iw}, and one can easily check that all arguments in \cite{iw} directly applied to $L_{n,2}$. One of the key point is that the definition of straight forks given here and in \cite{iw} are identical.
\end{remark}

To represent a homology class in $H_{m}(\tC_{n,m}, \partial \tC_{n,m};\Z)$, we use a similar geometric object.
A {\em noodle} is an oriented smooth embedded arc which begins at $d_{1}$ and ends at $d_{n}$. Let $\Sigma_{N}=\{ \{z_{1},\ldots,z_{m}\} \in C_{n,m} \: | \: z_{i} \in N \}$ and $\widetilde{\Sigma}(N)$ be the connected component of $\pi^{-1}(\Sigma_{N})$ that contains $\{\widetilde{d_{1}},\ldots,\widetilde{d_{m}}\}$. The $m$-dimensional submanifold $\widetilde{\Sigma}(N)$ defines an element of $H_{m}(\tC_{n,m}, \partial \tC_{n,m};\Z)$. By abuse of notation we use $N$ to represent both a noodle and its representing homology class $[\widetilde{\Sigma}(N)] \in H_{m}(\tC_{n,m}, \partial \tC_{n,m};\Z)$. A {\em standard noodle} $N_{i}$ is a noodle that encloses the $i$-th puncture point $p_{i}$ as shown in Figure \ref{fig:multifork} right.

The {\em noodle-fork pairing} is a homology intersection pairing 
\[ \langle\;\;,\;\;\rangle \co H_{m}(\tC_{n,m}, \partial \tC_{n,m};\Z)  \times H_{m}(\tC_{n,m};\Z) \rightarrow \Z[\sq^{\pm 1}, \st^{\pm 1}]. \]

For a noodle $N$ and multifork $\bF=\{F_{1},\ldots,F_{m} \}$, the noodle-fork pairing $\langle N,\bF \rangle$ can be computed as follows.
We assume that $N$ and $T(F_{i})$ transverse for each $i$. Then two submanifolds $\widetilde{\Sigma}(N)$ and $\widetilde{\Sigma}(\bF)$ also transverse. 
Let us take intersection points $z_{i} \in N \cap F_{i}$ for each $i$. Then a set of intersections $\bz= \{z_{1},\ldots,z_{m}\}$ corresponds to an intersection point of $\widetilde{\Sigma}(N)$ and $\widetilde{\Sigma}(\bF)$. 
This $\bz$ contributes the pairing $\langle N, \bF \rangle$ by a monomial $\varepsilon_{\bz}m_{\bz} = \varepsilon_{\bz} \sq^{a_{\bz}}\st^{b_{\bz}}$, where $\varepsilon_{\bz} \in \{\pm 1\}$ represents the sign of the intersection at $\bz$.

One can compute the monomial $m_{\bz}=\sq^{a_{\bz}}\st^{b_{\bz}}$ as follows.
Assume that on the noodle $N$, along the orientation of $N$ the intersection points $z_{1},\ldots,z_{m}$ appears in the order $z_{\tau(1)}, z_{\tau(2)},\ldots, z_{\tau(m)} $ (Here $\tau$ represents an appropriate permutation). For each $i$ let $\delta_{i}$ be a path connecting $z_{\tau(i)}$ and $d_{i}$ as in Figure \ref{fig:NFpairing}: $\delta_{i}$ goes back to near $d_{1}$ along a path parallel to $N$, then return to $d_{i}$ along a path parallel to $\partial D_{n}$. We choose these paths $\{\delta_{i}\}$ so that they are disjoint.

Now take three paths $A_{i}, B_{i}$ and $C_{i}$ in $D_{n}$ as follows (See Figure \ref{fig:NFpairing} left).
\begin{itemize}
\item $A_{i}$ is a path from $d_{i}$ to the branch point of $F_{i}$ along the handle of $F_{i}$.
\item $B_{i}$ is a path from the branch point to $z_{i}$ along the tine $T(F_{i})$.
\item $C_{i} = \delta_{\tau(i)}$.
\end{itemize}

Then the concatenation of the three paths \[ \{C_{1},\ldots,C_{m}\}\{B_{1},\ldots, B_{m}\}\{A_{1},\ldots, A_{m}\} \]
 defines a loop $l$ in $C_{n,m}$. The monomial $m_{\bz}$ is given by $\alpha(l) \in \langle \sq,\st \rangle$.

\begin{figure}[htbp]
 \begin{center}
\includegraphics[scale=0.5, width=100mm]{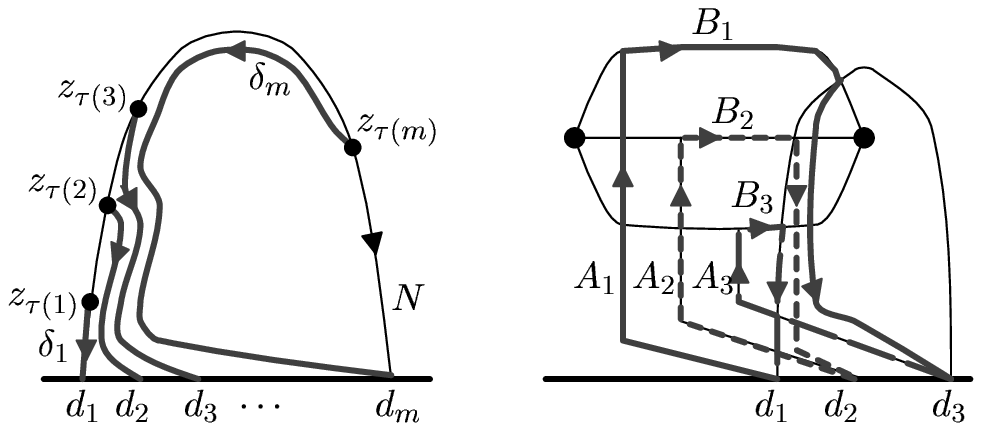}
   \caption{Computing Noodle-Fork pairing}
 \label{fig:NFpairing}
  \end{center}
\end{figure}

The key property of the noodle-fork pairing is that it detects the non-triviality of the geometrical intersections.
For monomials $\sq^{a}\st^{b}$ and $\sq^{a'}\st^{b'}$ we define the lexicographical ordering $\leqslant_{\sq,\st}$ by 
\[ \sq^{a}\st^{b} \leqslant_{\sq,\st} \sq^{a'}\st^{b'} \textrm{ \ \ \ if } a < a', \textrm{ or if } a=a' \textrm{ and } b \leqslant b'.\]  

The next lemma is a direct generalization of Bigelow's key lemma. The proof is a direct adaptation of Bigelow's proof.  

\begin{lemma}[Bigelow's key lemma, {\cite[Lemma 3.2 -- Claim 3.4]{big1}}]
\label{lemma:key}
Let $\bF=\{F_{1},\ldots,F_{m}\}$ be a multifork such that all tines are parallel.
Assume that a noodle $N$ have the minimal intersection with all tines $T(F_{i})$. Then all intersection points $\bz$ of $\widetilde{\Sigma}(N)$ with $\widetilde{\bF}$ which attain the $<_{\sq,\st}$-maximal monomial $m_{\bz}$ in $\langle N,\bF \rangle$ have the same sign $\varepsilon_{\bz}$.
\end{lemma}

This lemma proves that $L_{n,m}$ is faithful, as in Bigelow's argument.

\section{Quantum representation}
\label{sec:quantum}

In this section we review quantum representations and a theorem of Kohno that relates quantum and Lawrence's representation.

\subsection{Quantum $\sltwo$ and its modification $\mU$} 

We review a construction of the algebra $\mU$, a modification of $U_{q}(\sltwo)$ introduced in \cite{jk}.
As we mentioned earlier, for all conventions about $U_{q}(\sltwo)$ we follows \cite{jk}. To introduce $\mU$, we first recall standard facts on $U_{q}(\sltwo)$. For details, see \cite{kas}. 

Let $q \in \C^{*}=\C-\{0\}$ be a non-zero complex-valued parameter.
The quantum group (the quantum enveloping algebra of $\sltwo$) $U_{q}(\sltwo)$ is an Hopf algebra over $\C$ generated by $E,F^{(n)}$ $(n=1,2,\ldots)$, $K$ and $K^{-1}$ having the relations
\begin{equation}
\label{eqn:relation}
\left\{
\begin{array}{l}
 KK^{-1} = K^{-1}K=1, \;\;\;KEK^{-1}=q^{2}E , \\
K F^{(n)}K^{-1} = q^{-2n}F^{(n)},\;\;\; F^{(n)}F^{(m)} = \left[ \!\! \begin{array}{c} n+m \\ n \end{array} \! \!\right]_{q}F^{(n+m)}.
 \end{array}
\right.
\end{equation}
The coproduct $\Delta$ and antipode $S$ are given by 
\begin{equation}
\label{eqn:coprod}
\left\{
\begin{array}{l}
\Delta(K) = K \otimes K,\;\;\; \Delta (E) = E \otimes K + 1 \otimes E,\\
\Delta( F^{(n)}) = \sum_{j=0}^{n} q^{-j(n-j)} K^{j-n}F^{(j)} \otimes F^{(n-j)} \\
 S(K)=K^{-1}, S(E) = -EK^{-1},\; S(F^{(n)}) = (-1)^{n} q^{n(n-1)} K^{n}F^{(n)}.
\end{array}
\right.
\end{equation}

For a complex number $s \in \C$, the {\em Verma module of highest weight $s$} is an $U_{q}(\sltwo)$-module $V_{s}$ spanned by $\{v_{0},v_{1},\ldots \}$, where the action is given by
\begin{equation}
\label{eqn:module}
\left\{ 
\begin{array}{rll}
Kv_{j} &= & sq^{-2j}v_{j} \\
Ev_{j} &= & v_{j-1} \\
F^{(n)}v_{j} & = &\left( \left[ \!\begin{array}{c} n+j \\ j \end{array}\! \right]_{q}{\displaystyle \prod_{k=0}^{n-1}} (sq^{-k-j} - s^{-1}q^{k+j}) \right) v_{j+n}. 
\end{array}
\right. 
\end{equation}

Now we modify the algebra $U_{q}(\sltwo)$ to treat both the weight $s$ and the quantum parameter $q$ as variables.
Let $\mL = \Z[q^{\pm 1},s^{\pm 1}]$ be the ring of two-variable Laurent polynomial of integer coefficient. 

Let $\mU$ be an Hopf algebra over $\mL$ generated by $E,F^{(n)},K,K^{-1}$  having the the same relation (\ref{eqn:relation}) and the coproduct and antipode are defined by (\ref{eqn:coprod}). 
Let $V$ be a free $\mL$-module spanned by $\{v_{0},v_{1},\ldots \}$.
Then (\ref{eqn:module}) defines an $\mU$-module structure on $V$. We say an $\mU$-module $V$ a {\em generic Verma module}. 

From the universal $R$-matrix of $U_{q}(\sltwo)$, we get a linear representation of the braid groups
  \[\rho: B_{n} \rightarrow \GL(V^{\otimes n}) \]
defined by 
  \[ \rho( \sigma_{i} ) = \textsf{id}^{\otimes (i-1)}\otimes R \otimes \textsf{id}^{\otimes (n-i-1)}. \]  
Here $R: V \otimes V \rightarrow V \otimes V$ is given by

\begin{equation}
\label{eqn:R-matrix}
R(v_{i} \otimes v_{j}) = s^{-(i+j)} \sum_{n=0}^{i} F_{i,j,n}(q) \prod_{k=0}^{n-1}(sq^{-k-j}-s^{-1}q^{k+j}) v_{j+n}\otimes v_{i-n}.
\end{equation}
 where $F_{i,j,n}(q)$ is a Laurent polynomial of the variable $q$ defined by
\begin{equation}
\label{eqn:F}
 F_{i,j,n}(q) = q^{2(i-n)(j+n)} q^{\frac{n(n-1)}{2}} \left[\!
  \begin{array}{c}\! \!n+j \\ j \end{array} \!\!\right]_{q}. 
\end{equation}
 
 For $n>1$, define $\Delta^{(n)}:\mU \rightarrow \mU^{\otimes n}$ by
 \[ \Delta^{(2)}=\Delta, \Delta^{(n)}= (\Delta^{(n-1)}\otimes \textsf{id})\Delta. \]
$\mU$ acts on $V^{\otimes n}$ by $u\cdot x = (\Delta^{(n)}u) x$.
For $n$ and $m$, the {\em space of generic highest weight vectors} is 
 \[ V_{n,m} = \textrm{ker}\; (K- s^{n}q^{-2m}). \]
The {\em space of (generic) null vectors} $W_{n,m}$ is a subspace of $V_{n,m}$
which $E$ trivially acts,
\[ W_{n,m} = \textrm{ker}\, (E) \cap V_{n,m}. \] 
It is directly seen that both $V_{n,m}$ and $W_{n,m}$ are also $B_{n}$-representations.

$V_{n,m}$ is a free $\mL$-module of rank $d_{n+1,m}$. It is spanned by the vectors
\[ \{ v_{e_{1}} \otimes v_{e_{2}} \otimes \cdots \otimes v_{e_{n}} \: | \: \be=(e_{1},\ldots, e_{n}) \in E_{n+1,m}.\} \]
In this paper we use slightly modified basis of $V_{n,m}$.
For $\be \in E_{n+1,m}$ we define $v_{\be} \in V_{n,m}$ by
\[ v_{\be} =s^{\sum_{i=1}^{n} ie_{i}} v_{e_{1}} \otimes v_{e_{2}} \otimes \cdots \otimes v_{e_{n}}. \]

Using this basis $\{v_{\be}\}$, we get a matrix valued representation
\[ \rho^{V}_{n,m} \co B_{n} \rightarrow \GL( d_{n+1,m} ; \mL ) \]

Similarly, $W_{n,m}$ is a free $\mL$-module of rank $d_{n,m}$ \cite[Theorem 1]{jk}.
 A basis of $W_{n,m}$ was given in \cite{jk}, but here we use a slightly different basis following Kohno, which corresponds to the standard multifork basis of $H_{m}(\tC_{n,m};\Z)$ as we will see later.

Let $i \co E_{n,m} \rightarrow E_{n+1,m}$ be an injection defined by
\[ i((e_{1},\ldots,e_{n-1})) = (0,e_{1},\ldots,e_{n-1}). \]
Let $V'_{n,m}=\mL v_{0} \otimes V_{n-1,m} \subset V_{n,m}$. Then $\{v_{i(\be)}\}_{\be \in E_{n,m}}$ gives a basis of $V'_{n,m}$.

Define a map $\Phi \co V'_{n,m} \rightarrow W_{n,m}$ by
 \[ \Phi( v_{0}\otimes u) = \sum_{k=0}^{m} (-1)^{k} q^{2km - k(k+1)} s^{-k(n-1)} v_{k} \otimes E^{k}u. \]

$\Phi$ is an isomorphism of $\mL$-modules. By using the image of basis $\{\Phi(v_{i(\be)})\}_{\be \in E_{n,m}}$ as a basis of $W_{n,m}$, we get a matrix representation
\[ \rho^{W}_{n,m} \co B_{n} \rightarrow \GL( d_{n,m};\mL ). \]

We will sometimes regard the map $\Phi$ as a map $V'_{n,m} \rightarrow V_{n,m}$ in an obvious way, and regard $\Phi$ as a $(d_{n,m} \times d_{n+1,m})$-matrix $M_{\Phi}$
\[ M_{\Phi} \co V'_{n,m} \cong \mL^{d_{n,m}} \rightarrow \mL^{d_{n+1,m}} \cong V_{n,m} \]
 by using the basis $\{ v_{i(\be)} \}_{\be \in E_{n,m}}$ of $V'_{n,m}$ and the basis $\{v_{\be}\}_{\be \in E_{n+1,m}}$ of $V_{n,m}$.
 
For an $n$-braid $\beta$, using the matrices $M_{\Phi}$ and $\rho^{V}_{n,m}(\beta)$ we can compute the $(d_{n,m} \times d_{n,m})$-matrix $\rho^{W}_{n,m}(\beta)$ as follows.

Let us define $\rho^{\Phi}_{n,m}:V'_{n,m} \rightarrow V'_{n,m}$ by $\rho^{\Phi}_{n,m} = \Phi^{-1} \circ \rho^{W}_{n,m}\circ \Phi$.
Since we have identified the basis of $V'_{n,m}$ and $W_{n,m}$ by $\Phi$, the matrix $\rho^{W}_{n,m}(\beta)$ coincides with the matrix expression of $\rho^{\Phi}_{n,m}(\beta)$ with respect to the basis $\{v_{i(\be)} \}_{\be \in E_{n,m}}$. 
 
Let $\pi': V_{n,m} \rightarrow V'_{n,m}$ be the projection map. Then $\Phi^{-1}|_{W_{n,m}} =\pi'|_{W_{n,m}}$, so as a matrix we get an equality
\begin{equation}
\label{eqn:WV}
 \rho^{W}_{n,m}(\beta) = \pi' \circ \rho^{V}_{n,m}(\beta) \circ M_{\Phi}.
\end{equation}

The next lemma shows that our choice of basis behaves well with respect to the map $\Phi$, that is, $\Phi$ does not affect the degree of the variable $s$.

\begin{lemma}
\label{lemma:phi}
$M_{s}(M_{\Phi})=m_{s}(M_{\Phi})=0$. That is, each entry of the matrix $M_{\Phi}$ does not involve the variable $s$.
\end{lemma} 
\begin{proof}

First observe that by definition of the action of $\mU$ on $V_{n,m}$, we get
\[ E^{k} (v_{e_{1}} \otimes \cdots \otimes v_{e_{n-1}}) = \sum_{\bk \in E'_{n,k}} F_{\bk}(q) s^{-\sum_{i=1}^{n}i k_{i} + k(n-1)}  (v_{e_{1}-k_{1}} \otimes \cdots \otimes v_{e_{n-1}-k_{n-1}}) \]
where $E'_{n,k}$ is a subset of $E_{n,k}$ defined by
\[ E'_{n,k} =\{\bk= (k_{1},\ldots,k_{n-1}) \in E_{n,k} \: | \: k_{i} \leq e_{i} \}. \]
and $F_{\bk}(q)$ denotes a non-zero Laurent polynomial of $q$. 
For $\bk \in E'_{n,k}$, let 
\[ \be(\bk) = (k, e_{1}-k_{1},\ldots,e_{n-1}-k_{n-1}) \in E_{n+1,m}. \]

For $\be \in E_{n,m}$ we get
\begin{eqnarray*}
M_{\Phi} (v_{i(\be)}) & = & s^{m+ \sum_{i=1}^{n} ie_{i}} M_{\Phi}(v_{0} \otimes v_{e_{1}} \otimes \cdots \otimes v_{e_{n-1}}) \\
  & = & s^{m+ \sum_{i=1}^{n} ie_{i}} \sum_{k=0}^{m} F_{k}(q) s^{-k(n-1)} v_{k} \otimes E^{k} (v_{e_{1}} \otimes \cdots \otimes v_{e_{n-1}})\\
  & = & s^{m+ \sum_{i=1}^{n} ie_{i}} \sum_{k=0}^{m} F_{k}(q) s^{-k(n-1)} \\
&  &  \;\;\; \left( \sum_{\bk \in E'_{n,k}} F_{\bk}(q) s^{-\sum_{i=1}^{n} i k_{i} + k(n-1)} v_{k} \otimes v_{e_{1}-k_{1}} \otimes \cdots \otimes v_{e_{n-1}-k_{n-1}} \right)\\
  & = & \sum_{k=1}^{m} \sum_{\bk \in E'_{n,k}} F_{k,\bk}(q) s^{m+ \sum_{i=1}^{n}(ie_{i} - ik_{i}) } v_{k} \otimes v_{e_{1}-k_{1}} \otimes \cdots \otimes v_{e_{n-1}-k_{n-1}}\\
  & = & \sum_{k=1}^{m} \sum_{\bk \in E'_{n,k}} F_{k,\bk}(q) v_{\be(\bk)}
\end{eqnarray*}
where $F_{k}(q)$ and $F_{k,\bk}(q)$ represent non-zero Laurent polynomials of $q$.
Thus, each entry of the matrix $M_{\Phi}$ lies in the Laurent polynomial ring $\Z[q,q^{-1}]$. 
\end{proof}

\subsection{Kohno's theorem}

Next we review Kohno's theorem on Lawrence's representation and quantum representation. For details, see \cite{koh}. Although he did not explicitly state, he actually proved Jackson-Kerler's conjecture \cite[Conjecture 4]{jk}.

Kohno's approach was based on the theory of KZ-equations. To use theory of KZ-equation,
we regard $B_{n}$ as the fundamental group of so-called the complement braid arrangement and replace the configuration space $C_{n,m}$ with the complement of certain hyperplane arrangement. 

Let $X_{n}$ be the complement of the braid arrangement, that is,
\[ X_{n} = \C^{n} -\bigcup_{1\leq i<j \leq n} \textrm{Ker}\,(z_{i}-z_{j}) \]
and $Y_{n}$ be its quotient
\[ Y_{n} = X_{n} \slash S_{n}\]
where the symmetric group $S_{n}$ acts on $X_{n}$ as a permutation of coordinates. Then $\pi_{1}(Y_{n})\cong B_{n}$.

Let $\fg$ be a complex semi-simple Lie algebra (in this paper we only use the case $\sltwo(\C)$) and let $\{I_{\mu}\}$ be an orthogonal basis of $\fg$ with respect to the Cartan-Killing form. Put $\Omega=\sum_{\mu} I_{\mu}\otimes I_{\mu} \in \fg\times \fg$.

Take a $\fg$-module $V$ and denote by $\Omega_{ij}$ the action of on $\Omega$ on the $i$-th and $j$-th component of $V^{\otimes n}$.
The {\em Knizhnik-Zamolodchilov connection} ({\em KZ-connection}, in short) is an $\textrm{End}(V^{\otimes n})$-valued 1-form given by
\[ \omega_{KZ} = \frac{\hbar}{\pi \sqrt{-1} } \sum_{1\leq i<j \leq n} \Omega_{ij}\frac{dz_{i}-dz_{j}}{z_{i}-z_{j}}.\]
where $\hbar \in \C^{*}$ is a certain complex parameter.
(The parameter $\hbar$ corresponds to $\frac{\pi \sqrt{-1}}{\kappa}$ in Kohno's notation in \cite{koh}, and coincide with the one used in \cite{jk}. It corresponds to the parameter $\frac{h}{2}$ in Kassel's book \cite{kas}).

This is a flat connection of $Y_{n} \times V^{\otimes n}$, a trivial vector bundle over $Y_{n}$ with fiber $V^{\otimes n}$. 
By considering the monodromy representation of the flat connection, we get a linear representation
\[ \theta \co B_{n} = \pi_{1}(Y_{n}) \rightarrow \GL( V^{\otimes n}) \]

Now we consider the following special case of the monodromy representations.
Let $\lambda$ be a complex number and $M_{\lambda}$ be the Verma module of $\sltwo(\C)$ with highest weight $\lambda$:
$M_{\lambda}$ is a $\C$-vector space spanned by $\{ v_{j} \}_{j=0,1,\ldots }$ and the action of $\sltwo(\C)$ is given by
\[
\left\{
\begin{array}{lll}
 H v_{j}& = & (\lambda - 2j )v_{j} \\
 E v_{j}& = & v_{j-1}\\
 F v_{j} & = &(j+1)(\lambda - j ) v_{j+1}.
 \end{array}
 \right.
.
\]
Here $H,E,F$ are standard generators of $\sltwo(\C)$.

The {\em space of null vectors} is 
\[ N[n \lambda -2m] = \{v \in M_{\lambda}^{\otimes n} \: | \: H v = (n \lambda -2m) v, Ev=0 \}\]
$N(m\lambda -2m)$ have a basis $\{w_{\be}\}$ indexed by $\be \in E_{n,m}$,
\[ w_{\be} = \sum_{i=0}^{m} (-1)^{i} \frac{1}{\lambda ( \lambda -1) \cdots (\lambda -i) } F^{i}v_{0} \otimes E^{i}( F^{e_{1}}v_{0} \otimes F^{e_{2}}v_{0} \otimes \cdots \otimes F^{e_{n-1}}v_{0} ).\]

The KZ connection $\omega_{KZ}$ commutes with the diagonal action of $\sltwo(\C)$ on $M_{\lambda}^{\otimes n}$, so the monodromy of KZ connection gives a linear representation which takes value in $N[n\lambda - 2m]$,
\[ \rho_{KZ} \co B_{n} \rightarrow \GL ( N[ n\lambda- 2m]). \]

\begin{lemma}
The map $w_{\be} \rightarrow v_{i(\be)}$ $(\be \in E_{n,m})$ defines a $\C B_{n}$-module isomorphism between $N[n\lambda -2m]$ and $W_{n,m}|_{q=e^{\hbar}, s = e^{\hbar\lambda}}$.
\end{lemma}
\begin{proof}
This is a consequence of Drinfel'd-Kohno Theorem \cite{dri},\cite{koh1}, \cite[Theorem XIX.4.1]{kas} on monodromy representations and quantum representations. See \cite[Chapter XIX]{kas} for detail. 
\end{proof}

To give a homological interpretation of $N[n\lambda -2m]$, we use a slightly different description of Lawrence's representation.

Let $\pi_{n,m}:X_{n+m} \rightarrow X_{n}$ be the projection of the first $n$-coordinates: 
\[ \pi_{n,m}(z_{1},\ldots,z_{n},t_{1},\ldots,t_{m}) = (z_{1},\ldots, z_{n}).\] 
This is a fiber bundle over $X_{n}$, and let $X_{n,m}= \pi_{n,m}^{-1}(p_{1},\ldots,p_{n})$ be its fiber, which is the complement of hyperplane arrangement called the {\em discriminantal arrangement},
\[ X_{n,m} = \C^{m} - \left( \bigcup_{1\leq i<j \leq m} \textrm{Ker}\,(t_{i}-t_{j}) \cup \bigcup_{1\leq i \leq m, 1\leq l \leq n} \textrm{Ker}\,(t_{i}-p_{l}) \right). \]
Let
\[ Y_{n,m} = X_{n,m} \slash S_{m} \]
where the symmetric group $S_{m}$ acts on $X_{n,m}$ as a permutation of the coordinates.
Then 
\[ \pi_{n,m}: X_{n+m} \slash (S_{n}\times S_{m}) \rightarrow Y_{n} \]
be a fiber bundle over $Y_{n}$ whose fiber is $Y_{n,m}$.
By definition, there are natural inclusion $\iota: C_{n,m} \rightarrow Y_{n,m}$ which is a homotopy equivalence.

For complex numbers $\lambda$ and $\hbar$, we consider the specialization $\sq=e^{2 \hbar \lambda}$ and $\st = -e^{-2\hbar}$ to get a representation
\[r_{\lambda,\hbar}: \pi_{1}(Y_{n,m}) = \pi_{1}(C_{n,m}) \stackrel{\alpha}{\rightarrow} \langle \sq,\st \rangle  \rightarrow \C^{*}. \]
Let $\mcL_{\lambda,\hbar}$ be the associated local system on $Y_{n,m}$.

For $\be \in E_{n,m}$ let $\Delta_{\be}$ be the bounded chamber in $\R^{m}$ defined by 
\[ \Delta_{\be} =\{ (t_{1},\ldots,t_{m} ) \in \R^{m} \: | \: i <t_{m_{1}+\cdots + m_{i-1} +1} < \cdots t_{m_{1}+\cdots +m_{i-1}+m_{i}} < i+1\} \]
and $[\Delta_{\be}] \in H_{m}(Y_{n,m};\mcL_{\hbar,\lambda})$ be the homology class represented by the image of $\Delta_{\be}$.

\begin{proposition}[\cite{koh}]
\label{proposition:generic}
There is an open dense subset $U$ of $(\C^{*})^{2}$ such that  for $(\lambda,\hbar) \in U$, the following holds:
\begin{enumerate}
\item The set of bounded chambers $\{\Delta_{\be}\}_{\be \in E_{n,m}}$ is a basis of $H_{m}(Y_{n,m};\mcL_{\hbar,\lambda})$.
\item The map
\[ \iota_{*}: H_{m}(\tC_{n,m},\Z)|_{\sq=e^{2 \hbar \lambda},\st= -e^{-2\hbar}} \rightarrow H_{m}(Y_{n,m}, \mcL_{\lambda, \hbar})\]
induces an isomorphism of $\C B_{n}$-module. 
\end{enumerate}
\end{proposition}

\begin{remark}
In Proposition \ref{proposition:generic} (2), the signs are different from Kohno's paper \cite{koh}. This difference is due to the convention that we regard a positive direction of winding, that is, a positive orientation of the meridian of hypersurfaces as {\em clockwise direction}, which is different from usual conventions adapted in \cite{koh}.
\end{remark}
 
 We say two parameters $(\lambda,\hbar) \in (\C^{*})^{2}$ are {\em generic} if $(\hbar,\lambda) \in U$..
 By definition of the bounded chamber $\Delta_{\be}$ and standard multifork $\bF_{\be}$ the map $\iota$ sends the standard multifork basis $\{\bF_{\be}\}_{\be \in E_{n,m}}$ to the bounded chamber basis $\{\Delta_{\be}\}_{\be \in E_{n,m}}$.

Now Kohno's theorem is stated as follows. The idea of his proof is to express solutions of KZ equation (a horizontal section of KZ connection) as certain hypergeometric type integrals over bounded chamber $\Delta_{\be}$. See \cite{koh} for details.

\begin{theorem}[{\cite[Theorem 6.1]{koh}}]
For generic $(\lambda,\hbar)$, the map 
\[ \Psi: H_{m}(Y_{n,m};\mcL_{\lambda,\hbar}) \rightarrow N[n\lambda-2m] \] defined by $\Psi(\Delta_{\be})= w_{\be}$ is a $\C B_{n}$-module isomorphism.
\end{theorem}
  
 As a consequence, we are able to identify the Lawrence's representation $L_{n,m}$ and a quantum representation $W_{n,m}$.
 
\begin{corollary}[{ \cite[Conjecture 4]{jk} }]
\label{theorem:jkconjecture}
For $\beta \in B_{n}$, as a $(d_{n,m}\times d_{n,m})$-matrix, we have an equality
\[ L_{n,m}(\beta) = \rho^{W}_{n,m}(\beta)|_{s^{2}=\sq, -q^{-2}=\st} \]
\end{corollary}

\begin{proof}[Proof of Corollary]
We have observed that for generic $(\lambda,\hbar) \in )\C^{*})^{2}$, there are isomorphisms of $\C B_{n}$-modules
\begin{eqnarray*}
 H_{m}(\tC_{n,m};\C)|_{\sq=e^{2\hbar\lambda}, \st=-e^{-2\hbar} } & \cong & H_{m}(\widetilde{Y}_{n,m};\C)|_{\sq=e^{2\hbar\lambda}, \st=-e^{-2\hbar}} \cong [n\lambda -2m] \\
& \cong & W_{n,m}|_{s=e^{\hbar\lambda},q=e^{\hbar}}
\end{eqnarray*}
which preserve the basis indexed by $E_{n,m}$.
Since the set $U$ of generic parameters $(\lambda,\hbar)$ are open dense, by regarding $(\lambda,\hbar)$ as variable we get a desired result. 
\end{proof}

\section{The dual Garside length formula}
\label{sec:proof}
In this section we prove that representations $L_{n,m}$ and $\rho^{V}_{n,m}$ naturally and directly detect the dual Garside length.

\subsection{Primary calculations}

Before proving the dual Garside length formula, we need to compute the matrices for dual simple elements. 

We first calculate the matrix $\rho^{V}_{n,m}(\sigma_{i}\cdots \sigma_{j-2}\sigma_{j-1})$ for $1\leqslant i< j \leqslant n$.

For $1\leqslant i < j \leqslant n$ and $\be =(e_{1},\ldots, e_{n}) \in E_{n+1,m}$, let 
\[ E''_{+}=E''_{+}(\be,i,j) = \left\{\bk=(k_{1},\ldots,k_{n}) \in E_{n+1,m} \: \begin{array}{|ll}
k_{t}=0 & (t<i, j \leqslant t) \\
 k_{t} \leqslant e_{t} &(i\leqslant t < j) 
\end{array}
\right\}.
 \]
and for $\bk \in E''_{+}$, let 
\begin{eqnarray*}
 \be_{+}(\bk) & = & (e_{1},\ldots, e_{i-1}, e_{j}+ \sum_{t=i}^{j-1} k_{t}, e_{i}-t_{i},\ldots,e_{j-1}-t_{j-1},e_{j+1},\ldots,e_{n} ) \\
 & \in & E_{n+1,m}. 
\end{eqnarray*}

\begin{lemma}
\label{lemma:plus}
For $\be \in E_{n+1,m}$ let us write 
\[ \sigma_{i}\cdots \sigma_{j-2}\sigma_{j-1}(v_{\be}) = \sum_{\be' \in E_{n+1,m}} x_{\be'} v_{\be'}\;\;\; (x_{\be'} \in \Z[s^{\pm 1},q^{\pm1}] ) \]
Then 
\[ M_{s}(x_{\be'}) = m_{s}(x_{\be'}) = 0 \]
if $\be' \not \in \{ \be_{+}(\bk) \: | \: \bk \in E''_{+}(\be,i,j)\}$ and
\[ \left\{\begin{array}{l}
M_{s} (x_{\be'}) \leqslant -2 \sum_{t=i}^{j-1}(e_{t}-k_{t}) \\
m_{s} (x_{\be'}) \geqslant -2 \sum_{t=i}^{j-1}e_{t}
\end{array}
\right.\]
if $\be'=\be_{+}(\bk)$ for some $\bk \in E''_{+}(\be,i,j)$.
\end{lemma}
\begin{proof}
It is sufficient to prove lemma for the case $(j-i)=1$, since the general case $(j-i)>1$ is proved by applying the inequality for the case $(j-i)=1$ repeatedly.

By definition of the action of $B_{n}$ on $V^{\otimes n}$ (See (\ref{eqn:R-matrix}) ),
\begin{eqnarray*} \sigma_{j-1}(v_{e_{1}}\otimes \cdots \otimes v_{e_{n}}) & = & \! 
\sum_{k=0}^{e_{j-1}} s^{-(e_{j-1}+e_{j})} F_{e_{j-1},e_{j},k}(q) \prod_{t=0}^{k-1}(sq^{-t-j}-s^{-1}q^{t+j}) \\
& & \! (v_{e_{1}} \otimes \cdots \otimes v_{e_{j-2}} \otimes v_{e_{j}+k} \otimes v_{e_{j-1}-k}\otimes v_{e_{j+1}} \otimes \cdots \otimes v_{e_{n}})
\end{eqnarray*}
where $F_{e_{j-1},e_{j},k}$ is a non-zero Laurent polynomial of the variable $q$ defined by (\ref{eqn:F})
Thus, if $\be' = (e_{1},\ldots,e_{j-2},e_{j}+k, e_{j-1}-k, e_{j+1},\ldots,e_{n})$ then
\begin{eqnarray*}
x_{\be'}& = & s^{\sum_{t=1}^{n} te_{t}}   F_{e_{j-1},e_{j},k}(q)  s^{-(e_{j-1}+e_{j})}\prod_{t=0}^{k-1}(sq^{-t-j}-s^{-1}q^{t+j})  \\
 &  & \;\;\; \;\;\;\;\;\;\;\; s^{- \sum_{t=1}^{j-2} te_{t} - (j-1)(e_{j}+k) - j(e_{j-1}-k) - {\sum_{t=j+1}^{n} te_{t}}} \\
 & = &  F_{e_{j-1},e_{j},k}(q) \prod_{t=0}^{k-1}(sq^{-t-j}-s^{-1}q^{t+j})\\
& & \;\;\; \;s^{(j-1)e_{j-1} + je_{j} -(e_{j-1}+e_{j})  - (j-1)(e_{j}+k) - j(e_{j-1}-k)} \\
& = &
F_{e_{j-1},e_{j},k}(q)\prod_{t=0}^{k-1}(sq^{-t-j}-s^{-1}q^{t+j})  s^{-2e_{j-1} + k} 
\end{eqnarray*}
so \[
 M_{s}(x_{\be'}) = -2(e_{j-1}-k), \;\;\; 
 m_{s}(x_{\be'}) = -2 e_{j-1}.
\]
Otherwise, $x_{\be'}=0$ hence $M_{s}(x_{\be'}) = m_{s}(x_{\be'})=0$.
\end{proof}

We prove similar formula for the braid $(\sigma_{i}\cdots \sigma_{j-2}\sigma_{j-1})^{-1}$.

For $1\leqslant i < j \leqslant n$ and $\be =(e_{1},\ldots, e_{n}) \in E_{n+1,m}$, let 
\[ E''_{-}=E''_{-}(\be,i,j) = \left\{\bk=(k_{1},\ldots,k_{n}) \in E_{n+1,m} \: \begin{array}{|ll} k_{t}=0 & (t \leqslant i, j < t) \\ k_{t}  \leqslant e_{t} & (i < t \leqslant j)
\end{array} \right\} \]
and for $\bk \in E''_{-}$, let 
\begin{eqnarray*}
 \be_{-}(\bk) & = & (e_{1},\ldots, e_{i-1},e_{i+1}-t_{i+1},\ldots,e_{j}-t_{j},e_{i} + \sum_{t=i+1}^{j} k_{t}, e_{j+1},\ldots,e_{n} ) \\
 & \in & E_{n+1,m}.
 \end{eqnarray*}

By similar argument as in the proof of Lemma \ref{lemma:plus}, we get the corresponding results for $(\sigma_{i}\cdots \sigma_{j-2}\sigma_{j-1})^{-1}$.
 
\begin{lemma}
\label{lemma:minus}
For $\be \in E_{n+1,m}$ let us write 
\[ (\sigma_{i}\cdots \sigma_{j-2}\sigma_{j-1})^{-1}(v_{\be}) = \sum_{\be' \in E_{n+1,m}} x_{\be'} v_{\be'} \;\;\; (x_{\be'} \in \Z[s^{\pm 1}, q^{\pm 1}]) .\]
Then 
\[ M_{s}(x_{\be'}) = m_{s}(x_{\be'}) = 0 \]
if $\be' \not \in \{ \be_{-}(\bk) \: | \: \bk \in E''_{-}(\be,i,j)\}$ and
\[ \left\{\begin{array}{l}
M_{s} (x_{\be'}) \leqslant 2 \sum_{t=i+1}^{j}e_{t} \\
m_{s} (x_{\be'}) \geqslant 2 \sum_{t=i+1}^{j}(e_{t}-k_{t})
\end{array}
\right.\]
if $\be'=\be_{-}(\bk)$ for $\bk \in E''_{-}(\be,i,j)$.
\end{lemma}

By using the above calculations, we prove the following results on the matrix $\rho^{V}_{n,m}$. 

\begin{lemma}
\label{lemma:dualsimple_quantum}
If $\beta \in B_{n}$ is a dual-simple element, then 
\[ -2m \leqslant m_{s}(\rho^{V}_{n,m}(\beta) ) )\leqslant M_{s}(\rho^{V}_{n,m}(\beta) ) ) \leqslant 0.\]
\end{lemma}
\begin{proof}
Each dual simple element $x$ is written as $ x= a_{i_{k},j_{k}}^{-1}\cdots a_{i_{1},j_{1}}^{-1}\delta$. Thus, to prove lemma, it is sufficient to check the following two inequalities.
\begin{enumerate}
\item $-2m \leqslant m_{s}(\rho^{V}_{n,m}(\delta) ) M_{s}(\rho^{V}_{n,m}(\delta) ) \leqslant 0$.
\item $0 \leqslant m_{s}(\rho_{n,m}^{V}(a_{i,j}^{-1})) M_{s}(\rho_{n,m}^{V}(a_{i,j}^{-1}))\leqslant 2m$.
\end{enumerate}
 
The inequality (1) follows from Lemma \ref{lemma:plus}. 
Since 
\[ a_{i,j}^{-1} = (\sigma_{i}\cdots \sigma_{j-2}\sigma_{j-1})^{-1} (\sigma_{i+1} \cdots \sigma_{j-2}\sigma_{j-1}) \]
the inequality (2) follows from Lemma \ref{lemma:plus} and Lemma \ref{lemma:minus}. 
\end{proof}

The corresponding result for Lawrence's representation can be obtained from Lemma \ref{lemma:dualsimple_quantum} by using (\ref{eqn:WV}), Lemma \ref{lemma:phi} and Corollary \ref{theorem:jkconjecture}.

\begin{lemma}
\label{lemma:dualsimple_homological}
For $\beta \in B_{n}$, we have an inequality
\[ \frac{1}{2} m_{s}(\rho^{V}_{n,m}(\beta ) ) \leqslant m_{\sq}(L_{n,m}(\beta ) ) \leqslant M_{\sq}(L_{n,m}(\beta ) ) \leqslant \frac{1}{2} M_{s}(\rho^{V}_{n,m}(\beta ) ) . \]
In particular, if $\beta \in B_{n}$ is a dual-simple element, then 
\[ -m \leqslant m_{\sq}(L_{n,m}(\beta)) \leqslant M_{\sq}(L_{n,m}(\beta)) \leqslant m\]
holds.
\end{lemma}
\begin{proof}
Recall the equality (\ref{eqn:WV}), $\rho^{W}_{m,n} (\beta) = \pi' \rho^{V}_{m,n}(\beta) M_{\Phi}$. 
By Lemma \ref{lemma:phi}, the matrix $M_{\phi}$ does not involve the variable $s$,
\[ m_{s}(\rho^{V}_{m,n}(\beta) ) \leqslant m_{s}(\rho^{V}_{m,n}(\beta ) M_{\Phi} ). \]
Since the matrix $\pi' \rho^{V}_{m,n}(\beta) M_{\Phi}$ is a submatrix of $\rho^{V}_{m,n}(\beta) M_{\Phi}$,
\[ m_{s} ( \rho^{V}_{m,n}(\beta ) M_{\Phi}) \leqslant m_{s}( \pi' \rho^{V}_{m,n}(\beta ) M_{\Phi } )  = m_{s}(\rho^{W}_{m,n}(\beta ) ) \]
hence
\[  m_{s}(\rho^{V}_{m,n}(\beta) ) \leqslant m_{s}(\rho^{W}_{m,n}(\beta ) ). \]
By Corollary \ref{theorem:jkconjecture}, $L_{m,n}(\beta) = \rho^{W}_{m,n}(\beta)|_{s^{2}=\sq, -q^{-2}=\st}$ so we conclude
\[ \frac{1}{2} m_{s}(\rho^{V}_{m,n}(\beta)) \leqslant m_{\sq}(L_{n,m}(\beta)) \]
The inequality between $M_{s}$ and $M_{\sq}$ is proved in a similar way.
\end{proof}

\subsection{Dual Garside length formula for Lawrence's representation}
\label{sec:Lformula}

\begin{theorem}
\label{theorem:L}
Let $L_{n,m}$ $(m>1)$ be Lawrence's representation and $\beta \in B_{n}$. Then,
\begin{enumerate}
\item $m\sup_{\Sigma^{*}}(\beta) = -m_{\sq}(L_{n,m}(\beta))$.
\item $m\inf_{\Sigma^{*}}(\beta) = M_{\sq}(L_{n,m}(\beta))$.
\item $m\,l_{\Sigma^{*}}(\beta) = \max \{0,-m_{\sq}(L_{n,m}(\beta))\} - \min\{0,-M_{\sq}(L_{n,m}(\beta))\}$. 
\end{enumerate}
\end{theorem} 

\begin{proof}

The proof is almost identical with the proof of the LKB representation \cite{iw}, the case $m=2$. For details, see \cite{iw}.

The key notion is a certain labeling of curve diagram of braids, called the {\em wall-crossing labeling}. This serves as a bridge  between dual Garside length and the homology of local system coefficients. See \cite{iw} for precise definition of curve diagram and wall-crossing labeling.

Let $\beta$ be an $n$-braid, $\bS$ be a straight multifork and $N$ be a standard noodle. As in the LKB representation case, we are able to relate the degree of the variable $\sq$ in the noodle-fork pairing $\langle N, \beta(\bS) \rangle$ and the wall-crossing labeling:
For each intersection point $\bz=\{z_{1},\ldots,z_{m}\}$ of $\widetilde{\Sigma}(N)$ and $\widetilde{\Sigma}(\beta(\bS))$, one can compute $a_{\bz}$, the degree of the variable $\sq$ in the monomial $m_{\bz}=\sq^{a_{\bz}}\st^{b_{\bz}}$ from the wall-crossing labeling of curve diagrams (See \cite[Lemma 4.1]{iw}).

On the other hand, by \cite[Theorem 3.3]{iw}, we have shown that the {\em minimal} value of the wall-crossing labeling is equal to $- \sup_{\Sigma_{*}}(\beta)$. (Since we adapted a convention that $\sigma_{i}$ is right-handed twists which is opposite to \cite{iw}, the role of the maximal and the minimal are interchanged).

Now we choose a straight multifork $\bS$ and a standard noodle $N$ so that the following three conditions holds.
\begin{enumerate}
\item $\beta(\bS)$ and $N$ have the minimal intersection.
\item $\beta(\bS)$ contains a subarc $\alpha$ that attains the minimal wall-crossing labeling of the curve diagram of $\beta$.
\item $N \cap \alpha \neq \phi$. 
\end{enumerate}

Then, thanks to Lemma \ref{lemma:key} and the relationships among noodle-fork pairing, wall-crossing labeling, and the dual Garside length mentioned above, we get an inequality
\[ m_{q}(\langle N, \beta(\bS) \rangle) \leqslant - m \sup\!{}_{\Sigma^{*}}( \beta ) -m. \]
(See \cite[Lemma 5.2]{iw}, for the LKB representation case). 

Let us write $\beta(\bS)$ as the linear combination of the standard multifork basis
\[ \beta( \bS)=\sum_{\be \in E_{n,m}} x_{\be} \bF_{\be} \;\;\; (x_{\be} \in \Z[ \sq^{\pm 1},\st^{\pm 1}]) .\]
Then we have an equality
\[ \langle N, \beta (\bS) \rangle = \sum_{\be \in E_{n,m}} x_{\be} \langle N, \bF_{\be} \rangle . \]
On the other hand, by direct calculations for a standard multifork $\bF$,
\[ m_{\sq}(\langle N, \bF \rangle) \geqslant -m \]
holds.
Hence we conclude  
\[ \min_{\be \in E_{n,m}} \{ m_{\sq}(x_{\be})\} -m \leqslant
m_{\sq}(\langle N,\beta(\bS) \rangle ) \leqslant  -m \sup\!{}_{\Sigma^{*}}( \beta) -m . \]
Recall that a straight multifork $\bS$ is a standard multifork. So $x_{\be}$ is an entry of the matrix $L_{n,m}(\beta)$, hence we get an inequality
\[ m_{\sq}(L_{n,m}(\beta)) \leqslant  \min_{\be \in E_{n,m}} \{ m_{\sq}(x_{\be})\} \leqslant -m \sup \!{}_{\Sigma^{*}}(\beta) . \]

On the other hand, by lemma \ref{lemma:dualsimple_homological} 
\[ -m \leqslant m_{\sq}(L_{n,m}(x) ). \]
holds for each dual-simple element $x \in [1,\delta]$.
Hence we get the converse inequality
\[ -m \sup\!{}_{\Sigma^{*}}(\beta)\leqslant  m_{\sq}(L_{n,m}(\beta)), \] 
and conclude that
\[ m_{\sq}(L_{n,m}(\beta)) = -m \sup\!{}_{\Sigma^{*}} (\beta). \]
The proof of (2) is similar, and (3) follows from Proposition~\ref{prop:dGlength}.
\end{proof}

\subsection{Dual Garside length formula for quantum representation}

Finally we prove that the representation $\rho^{V}_{n,m}$ also naturally and directly detects the dual Garside length as Lawrence's representation does.

\begin{theorem}
\label{theorem:Q}
Let $\rho^{V}_{n,m}$ be the generic highest weight vectors and $\beta \in B_{n}$. Then, 
\begin{enumerate}
\item $- 2m \sup_{\Sigma^{*}}(\beta) =  m_{s}(\rho^{V}_{n,m}(\beta))$.
\item $- 2m \inf_{\Sigma^{*}}(\beta) = M_{s}(\rho^{V}_{n,m}(\beta))$.
\item $2m \,l_{\Sigma^{*}}(\beta) = \max\{ 0, -m_{s}(\rho^{V}_{n,m}(\beta)) \} - \min \{0,-M_{s}(\rho^{V}_{n,m}(\beta))\}$. 
\end{enumerate}
\end{theorem}
\begin{proof}
As we have seen in Lemma \ref{lemma:dualsimple_quantum}, for a dual simple element $x$ of $B_{n}$ we have an inequality
\[ -2m \leqslant m_{s}(\rho^{V}_{n,m}(x)) \leqslant 0\]
Hence we get   
\[ -2m \sup\!{}_{\Sigma^{*}}(\beta) \leqslant m_{s}(\rho^{V}_{n,m}(\beta)). \]

On the other hand, by Lemma \ref{lemma:dualsimple_homological} and Theorem \ref{theorem:L} we get the converse inequality
\[ m_{s}(\rho^{V}_{n,m}(\beta)) \leqslant 2m_{\sq}(L_{n,m}(\beta)) = -2m \sup\!{}_{\Sigma^{*}}(\beta). \]
so we obtain the desired equality
\[ -2m \sup\!{}_{\Sigma^{*}}(\beta) =  m_{s}(\rho^{V}_{n,m}(\beta) ). \]
The proof of (2) is similar, and (3) follows from Proposition \ref{prop:dGlength}.
\end{proof}

\end{document}